\documentclass[11pt]{amsart}

 \usepackage[T1]{fontenc}
 \usepackage[utf8]{inputenc}
 \usepackage[english]{babel}
 \usepackage{graphicx}
 \usepackage{geometry}

 \usepackage{amsmath,amsthm,amsfonts,amssymb}
 \usepackage{enumerate}
 \usepackage{color}
 \usepackage{fancyhdr}
 \usepackage[titletoc]{appendix}
 \usepackage{hyperref}
 \usepackage{tikz}
 \usepackage{tikz-cd}
 \usepackage{caption}
 \usetikzlibrary{matrix,arrows,decorations.pathmorphing, positioning}

\newtheorem{theo}{Theorem}[section]

\newtheorem{lemma}[theo]{Lemma}
\newtheorem{prop}[theo]{Proposition}
\theoremstyle{definition}
\newtheorem{defi}[theo]{Definition}
\theoremstyle{definition}

\newtheorem{example}[theo]{Example}
\newtheorem*{notat}{Notation}

\input{xy}
\xyoption{all}
\xyoption{poly}
\usepackage[all]{xy}

 \newcommand\R{{\mathbb{R}}}
 
 \newcommand\Z{{\mathbb{Z}}}

 \newcommand\F{{\mathbb{F}}}

 \DeclareMathOperator{\sys}{sys}
 \DeclareMathOperator{\area}{area}

 \DeclareMathOperator{\Hom}{Hom}

\title{Simplicial complexity of surface groups and systolic area}
   \author{Eugenio Borghini}
   \author{El\'ias Gabriel Minian}
   \address{Departamento  de Matem\'atica - IMAS\\
 FCEyN, Universidad de Buenos Aires. Buenos Aires, Argentina.}
\email{eborghini@dm.uba.ar ; gminian@dm.uba.ar}

\thanks{Researchers of CONICET. Partially supported by grant UBACyT 20020160100081BA}

\subjclass[2010]{53C23, 57M20, 57Q15, 57N16.}

\keywords{Systolic area, simplicial complexity, surface groups.}

\begin{document}

   \begin{abstract}
  The simplicial complexity is an invariant for finitely presentable groups that was recently introduced by Babenko, Balacheff and Bulteau to study systolic area. The simplicial complexity $\kappa(G)$ was proved to be a good approximation of the systolic area $\sigma(G)$ for large values of $\kappa(G)$. In this paper we compute the simplicial complexity of all surface groups (both in the orientable and in the non-orientable case). This settles a problem raised by Babenko, Balacheff and Bulteau. We also prove that $\kappa(G\ast \Z)=\kappa(G)$ for any surface group $G$. This provides the first partial evidence in favor of the conjecture of the stability of the simplicial complexity under free product with free groups. The general stability problem, both for simplicial complexity and for systolic area, remains open. 
   \end{abstract}

   \maketitle

\section{Introduction}

Let $X$ be a connected, non-simply connected finite simplicial complex of dimension 2. Given a piecewise smooth Riemannian metric $g$ on $X$, the \textit{systole} $\sys(X,g)$ of $(X,g)$ is the length of the shortest non-contractible loop in $X$. The \textit{systolic area} $\sigma(X)$ of $X$ is defined as
\[
	\sigma(X) := \inf_{g} \frac{\area(X,g)}{\sys(X,g)^{2}},
\]
where the infimum is taken over all piecewise smooth Riemannian metrics on $X$. In \cite{Gr2}, Gromov defined the systolic area of a finitely presentable group $G$ as 
\[
	\sigma(G) := \inf_{X} \sigma(X),
\]
where the infimum is taken over the finite 2-dimensional simplicial complexes $X$ with fundamental group isomorphic to $G$. Gromov proved that $\sigma(G)$ is strictly positive unless $G$ is free (see \cite[Theorem 6.7.A]{Gr1}) and posed the problem of estimating the asymptotic behavior of the size of the set ${\mathcal G}_{\sigma}(C)$ of isomorphism classes of groups $G$ for which $\sigma(G) \leq C$ \cite{Gr2}. By combining topological and geometric techniques,   Rudyak and Sabourau proved in \cite{RS} an exponential upper bound to the size of ${\mathcal G}_{\sigma}(C)$, and considerably improved Gromov's lower bound for the systolic area of a non-free group. In \cite{RS} they raised the following question concerning the stability of the systolic area under free product with free groups: is $\sigma(G \ast \Z) = \sigma(G)$ for any non-free group $G$? (see \cite[Question 1.2]{RS}). The computation of $\sigma(G \ast T)$ for non-free groups $G$ and $T$ is a problem formulated by Gromov in (\cite[p.337]{Gr2}).

In \cite{BBB} Babenko, Balacheff and Bulteau introduced the \textit{simplicial complexity} of a finitely presentable group, which is a combinatorial invariant that approximates the systolic area. The simplicial complexity $\kappa(G)$ of a finitely presentable group $G$ is the minimum number of $2$-simplices of a simplicial complex of dimension $2$ with fundamental group isomorphic to $G$ (see \cite[Definition 2.1]{BBB}). Here, the number of 2-simplices of a simplicial complex of dimension $2$ may be thought of as a combinatorial avatar for the area. They showed that the simplicial complexity of $G$ constitutes a fairly good approximation to the systolic area of $G$ for large values of $\kappa(G)$. More concretely, given $\varepsilon > 0$,
\[
 2\pi \sigma(G) \leq \kappa(G) \leq \sigma(G)^{(1+\varepsilon)}
\]
for a group $G$ of free index 0, whenever $\kappa(G)$ is big enough (see \cite[Theorem 1.2]{BBB} for the precise statement). Using this comparison result, they provided a quite satisfactory answer to Gromov's question on the size of ${\mathcal G}_{\sigma}(C)$, improving the upper bound given in \cite{RS} (see \cite[Theorem 1.1]{BBB}). In view of the close relation between $\sigma(G)$ and $\kappa(G)$, they asked whether the equality $\kappa(G \ast \Z) = \kappa(G)$ holds for any group $G$ (see \cite[Section 2]{BBB}). Another problem posed in \cite{BBB} is the exact computation of the value of $\kappa$ for some groups (\cite[Examples 1,2]{BBB}). To the best of our knowledge, the only groups for which the value of $\kappa$ is known are the fundamental group of the projective plane $\Z_{2}$, of the torus $\Z \oplus \Z$, of the Klein bottle, and $\Z_3$, computed in \cite{Bul} mainly by an exhaustive analysis.

The purpose of this article is to provide some partial answers to the problems described above. In first place, we compute the simplicial complexity of the fundamental groups of all non-simply connected closed surfaces settling in the affirmative a conjecture from \cite{Bul} (here by closed surface we mean a compact connected smooth 2-manifold, orientable or non-orientable, without boundary).
\begin{theo}\label{kappa_surf}
Let $S$ be a non-simply connected closed surface. Then, $\kappa(\pi_1(S)) = \delta(S)$, where $\delta(S)$ is the minimum number of 2-simplices in a triangulation of $S$.
\end{theo}
The numbers $\delta(S)$ were completely computed by Ringel \cite{Rin}, for non-orientable surfaces, and by Jungerman and Ringel \cite{JR} in the orientable case (see Theorem \ref{jungerman} below).

We also show that $\kappa(G \ast \Z) = \kappa(G)$ for any surface group $G$. Here by a surface group we mean the fundamental group of any non-simply connected closed surface $S$. This provides some (admittedly partial) evidence in favor of the validity of $\kappa(G \ast \Z) = \kappa(G)$ for any group $G$. Actually, we prove the following stronger estimate.
\begin{theo}\label{kappa_surf_free_prod}
Let $G$ be a surface group and $T$ a finitely presentable group. Then, $\kappa(G \ast T) \geq \kappa(G)$. In particular, $\kappa(G \ast \Z) = \kappa(G)$.
\end{theo}
To prove these results, we derive a sharp lower bound of $\kappa(G \ast T)$ for any group $G$ whose cohomology ring satisfies a certain regularity property. The proof of this estimate is based on rather elementary techniques that we developed previously in \cite{BM}.

The outline of the article is as follows. In Section 2, we recall the pertinent definitions and results from \cite{BM} and prove our main technical result. In Section 3, we use the main lemma to estimate the mentioned lower bound on $\kappa(G \ast T)$. Finally, in Section 4 we prove Theorems \ref{kappa_surf} and \ref{kappa_surf_free_prod}. For almost all closed surfaces ({\it non-exceptional} surfaces in the vocabulary of \cite{BM}), these results will follow straightforwardly from the lower bound provided by Theorem \ref{lower_bound}. In contrast, for handling the {\it exceptional} cases some ad-hoc arguments will be required (cf. \cite[Section 3]{BM}).

\section{Main technical result}

In this section we prove a central technical lemma that will allow us to give a lower bound of the simplicial complexity for groups of the form $G \ast T$, where $G$, $T$ are finitely presentable groups and the cohomology ring of $G$ satisfies a certain regularity property. By the cohomology ring of a group $G$ we will understand its cohomology as a discrete group, i.e. the cohomology of an Eilenberg-MacLane space $K(G,1)$. We will work with reduced (co)homology and the coefficient ring for (co)homology groups will be $\F_2$. Throughout the article, $G$ and $T$ will denote finitely presentable groups, and all the simplicial complexes that we deal with will be finite. We recall first the definition of property (A) from \cite[Definition 2.3]{BM}.

\begin{defi}\label{property}
Let $X$ be a topological space. We say that the cohomology ring $H^{*}(X)$ (with coefficients in $\F_2$) satisfies \textit{property (A)} if for every non-trivial $\alpha$ in $H^{1}(X)$, there exists $\beta \in H^{1}(X)$ such that $\alpha \cup \beta$ is non-trivial in $H^{2}(X)$.
\end{defi}

We will say that the cohomology ring of a group satisfies property (A) whenever the cohomology ring of a $K(G,1)$ space does.

\begin{example}\label{rem_surfaces}
Any surface group (orientable or non-orientable) satisfies property (A) by Poincaré Duality. 

More generally, any one-relator group $G$ with $H_2(G) =\F_2$ and with a non-degenerate cup product form $H^{1}(G) \times H^{1}(G) \to H^{2}(G) = \F_2$, satisfies property (A). Using the computation of the cohomology ring of one-relator groups of \cite{Ra}, one may obtain an infinite amount of examples of such groups. As concrete examples, the Baumslag-Solitar groups $BS(m,n)$ satisfy property (A) whenever $m$ and $n$ are odd.
\end{example}

Let $K$ be a connected simplicial complex of dimension $2$ such that $\pi_1(K) = G \ast T$. Observe that $K$ is the 2-skeleton of an aspherical CW-complex $X$ (possibly infinite dimensional). Since the fundamental group of $X$ is isomorphic to $G \ast T$, $X$ is an Eilenberg-MacLane space $K(G \ast T, 1)$. By a theorem of Whitehead (see for example \cite[Theorem 7.3]{B}), $X$ is homotopy equivalent to a space of the form $K_G \vee K_{T}$, the wedge union of a $K(G,1)$ space and a $K(T,1)$ space. Informally speaking, our first objective is to obtain a subcomplex of $K$ in which we have killed all the $2$-dimensional homology classes of $K$ that do not correspond to classes in $H_2(G)$. We start by giving a definition.

\begin{defi}
Let $X$ be a CW-complex of dimension at least $2$, together with a homotopy equivalence $h : X \to K_G \vee K_{T}$, where $K_G$, $K_{T}$ are defined as above. Assume further that its $2$-skeleton $X^{(2)}$ is a finite simplicial complex. 
Let $M \leq X^{(2)}$ be a (simplicial) subcomplex satisfying the following properties:
\begin{enumerate}
	\item The inclusion $i: M \hookrightarrow X$ induces isomorphisms $i_{*}:H_{n}(M) \to H_{n}(X)$ for $n < 2$.
	\item The composition $H_2(M) \xrightarrow[]{i_{*}} H_2(X) \equiv H_2(K_G)\oplus H_2(K_T)= H_2(G) \oplus H_2(T) \xrightarrow[]{p} H_2(G)$ is an epimorphism, where $p$ is the projection and the isomorphism $H_2(X) \equiv H_2(K_G) \oplus H_2(K_T)$ is induced by $h$.
\end{enumerate}
We will say that such a subcomplex $M$ is \textit{homologically $G$-full} with respect to $h$, or simply \textit{homologically $G$-full} if the homotopy equivalence $h$ is clear from the context.
\end{defi}

The next result says, roughly, that we can kill the ``extra'' homology classes in $H_2(X)$ one at a time. The result is inspired in \cite[Proposition 2.7]{BM}.

\begin{lemma}\label{kill_gamma}
Let $X$ be a CW-complex of dimension at least $2$ homotopy equivalent to a space of the form $K_G \vee K_{T}$ and such that its $2$-skeleton $X^{(2)}$ is a finite simplicial complex. Let $M \leq X^{(2)}$ be a homologically $G$-full subcomplex. If $\dim H_2(M) > \dim H_2(G)$, there exists a 2-simplex $\sigma \in M$ such that $M \setminus \sigma$ is homologically $G$-full. Moreover, $\dim H_2(M \setminus \sigma) = \dim H_2(M) - 1$.
\end{lemma}

\begin{proof}
Since by hypothesis $\dim H_2(M) > \dim H_2(K_G)$ there is a non-trivial class $B$ in the kernel of the linear map $p \circ i_{*}: H_2(M) \to H_2(K_G)$. Let $\sigma$ be a 2-simplex of $M$ in the support of $B$. The topological boundary $\partial \sigma$ viewed as a chain in $C_1(M \setminus \sigma)$ is the boundary of the $2$-chain $B - \sigma$. Hence the inclusion induces the zero morphism $ H_1(\partial \sigma) \to H_1(M\setminus{\sigma})$. It follows that the inclusion $M \setminus \sigma \hookrightarrow M$ induces isomorphisms $H_n(M \setminus \sigma) \to H_n(M)$ for $n < 2$. It remains to verify the surjectivity of $p \circ j_{*}: H_2(M \setminus \sigma) \to H_2(K_G)$, where $j$ is the inclusion $j: M\setminus \sigma \hookrightarrow X$. Let $[Z]$ be a class in $H_2(K_G)$. By hypothesis, there is some class $C \in H_2(M)$ such that $p \circ i_{*} [C] = [Z]$. If $\sigma$ does not belong to the support of $C$, when viewed as a class in $H_2(M \setminus \sigma)$ we have $p \circ j_{*} [C] = [Z]$. In the other case, the 2-chain $C + B$ is a well defined 2-cycle in $M \setminus \sigma$ and $p \circ j_{*}[C + B] = p \circ i_{*}[C] + p \circ i_{*}[B] = p \circ i_{*}[C] = [Z]$. Hence, in any case $p \circ j_{*}: H_2(M \setminus \sigma) \to H_2(K_G)$ is an epimorphism. The fact that $\dim H_2(M \setminus \sigma) = \dim H_2(M) - 1$ follows immediately from the Euler characteristic, since $\chi(M \setminus \sigma) = \chi(M) - 1$.
\end{proof}

\begin{notat} Given a finitely presentable group $G$, we will denote by $\chi_{\leq 2}(G)$ the \textit{2-truncated} Euler characteristic of $G$, that is $\chi_{\leq 2}(G) := \dim H_2(G) - \dim H_1(G) + \dim H_0(G)$.
\end{notat}

Recall that a simplex $\sigma$ of a simplicial complex $K$ is a \textit{free face} of $K$ if there is a unique simplex $\tau \in K$ containing $\sigma$ properly. In that case, we say that there is an (elementary) collapse from $K$ to the subcomplex $L$ obtained from $K$ by removing the free face $\sigma$ together with $\tau$. Note that the inclusion $L \subseteq K$ is, in particular, a strong deformation retract.

\begin{lemma}\label{main_tech}
Let $K$ be a (finite) connected simplicial complex of dimension 2 with fundamental group isomorphic to $G \ast T$, and suppose that the cohomology ring of $G$ satisfies property (A). Then, there is another simplicial complex $L$ of dimension at most 2 with no more 2-simplices than $K$ such that $\chi(L) \leq \chi_{\leq 2}(G)$, $\dim H_2(L) = \dim H_2(G)$ and every edge of $L$ is the face of at least two 2-simplices.
\end{lemma}

\begin{proof}
Let $X$ be an Eilenberg-MacLane space $K(G \ast T, 1)$ such that $X^{(2)}=K$. Then there is a map $i: K \to K_G \vee K_{T}$ inducing an isomorphism in $H_n$ for $n = 0,1$ and an epimorphism in $H_2$, where $K_G$ and $K_{T}$ are respectively a $K(G,1)$ and a $K(T, 1)$ space as before. Since the projection $H_2(X) \equiv H_2(G) \oplus H_2(T) \to H_2(G)$ is surjective, $K$ is a homologically $G$-full subcomplex of $X$. By applying inductively Lemma \ref{kill_gamma}, we obtain a subcomplex $M$ of $K$ that is homologically $G$-full and such that $\dim H_2(M) = \dim H_2(G)$. After collapsing the free faces of $M$, we may assume that $M$ has no edge that is the face of a unique 2-simplex. Suppose there is a maximal edge $e = \{a,b\}$ in $M$ (otherwise we are done, since we may take the desired complex $L$ as $M$). If there is no path between $a$ and $b$ in $M \setminus e$, the quotient $M / e$ has a natural structure of simplicial complex with one less maximal edge than $M$. If on the contrary, $a$ and $b$ are joined by some path in $M \setminus e$, $M$ is homotopy equivalent to a CW-complex of the form $Z \vee S^{1}$, where $Z$ is a simplicial complex of dimension at most 2. After applying, if needed, finitely many of these moves, we get a CW-complex of the form $L \vee_{i=1}^{m} S^{1}$ homotopy equivalent to $M$, where $L$ is a simplicial complex with no more 2-simplices than $M$ (hence than $K$) in which every edge is the face of at least two 2-simplices. It remains to verify the bound on the Euler characteristic of $L$. Since $L \vee_{i=1}^{m} S^{1}$ is homotopy equivalent to $M$, clearly
\[
	\chi(M) = \chi(L \vee_{i=1}^{m} S^{1}) = \chi(L) - m.
\]
On the other hand, by construction $\chi(M) = \chi_{\leq 2}(G) - \dim H_1(T)$, since $\dim H_2(M) = \dim H_2(G)$ and the first homology group of $M$ is isomorphic to $H_1(K_G \vee K_{T}) = H_1(G) \oplus H_1(T)$. Now, by composing with the equivalence $L \vee_{i=1}^{m} S^{1} \simeq M$ we obtain a map $f: L \vee_{i=1}^{m} S^{1} \to K_G \vee K_{T}$ which induces isomorphism in $H_n$ for $n = 0,1$ and such that $p \circ f_{*}: H_2(L \vee_{i=1}^{m} S^{1}) = H_2(L) \to H_2(G)$ is an epimorphism. In particular, dualizing we get an isomorphism
\[H^{1}(K_G \vee K_{T}) = H^{1}(G) \times H^{1}(T) \to H^{1}(L \vee_{i=1}^{m} S^{1}) = H^{1}(L) \times H^{1}(\vee_{i=1}^{m} S^{1}).\]
Let $(0,a) \in H^{1}(L) \times H^{1}(\vee_{i=1}^{m} S^{1})$ be a non-trivial class and suppose that $(\alpha, \delta) \in H^{1}(G) \times H^{1}(T)$ is the unique class such that $f^{*}(\alpha, \delta) = (0,a)$. We claim that $\alpha = 0$. Indeed, suppose that it was not the case. Then, since the cohomology ring of $G$ satisfies property (A) there is a class $\beta \in H^{1}(G)$ with $\alpha \cup \beta \neq 0$. Consider the class $f^{*}( (\alpha, \delta) \cup (\beta, 0) ) = f^{*}(\alpha \cup \beta, 0) \in H^{2}(L) = H^{2}(L) \times H^{2}(\vee_{i=1}^{m} S^{1})$. It is non-trivial: take a class $\lambda \in H_2(G)$ such that $(\alpha \cup \beta) \lambda \neq 0$ (here we use the identification $H^{2}(G) = \Hom(H_2(G), \F_2)$). Since $M$ is homologically $G$-full, there is some class $\gamma \in H_2(L) \equiv H_2(M)$ such that $f_{*}(\gamma) = (\lambda, \eta)$, for some $\eta \in H_2(T)$. Then,
\[
	f^{*}(\alpha \cup \beta, 0)\gamma = (\alpha \cup \beta, 0) f_{*}(\gamma) = (\alpha \cup \beta, 0) (\lambda, \eta) \neq 0.
\]
On the other hand, from the identity
\[
	f^{*}( (\alpha, \delta) \cup (\beta, 0) ) = (0,a) \cup f^{*}(\beta, 0) = 0
\]
we obtain a contradiction, proving the claim. We conclude that the inverse of $f^{*}: H^{1}(G) \times H^{1}(T) \to H^{1}(L) \times H^{1}(\vee_{i=1}^{m} S^{1})$ restricts to give a monomorphism $H^{1}(\vee_{i=1}^{m} S^{1}) \to H^{1}(T)$. Hence, $m \leq \dim H_1(T)$ and, since $\chi(L) = \chi_{\leq 2}(G) - (\dim H_1(T) - m)$, the result follows.
\end{proof}

\section{The lower bound}

This section is devoted to the proof of the announced lower bound on the simplicial complexity for groups of the form $G \ast T$, where $G$, $T$ are finitely presentable groups and $G$ satisfies property (A).
\par We begin by fixing some notations (see \cite[Section 2]{BM}).

\begin{notat}
Let $k \in \Z$, $k \leq 2$. We denote by $\rho(k)$ the integer number defined as
\[
	\rho(k) := \left\lceil \frac{7 + \sqrt{49-24 k}}{2} \right\rceil.
\]
By abuse of notation, if $K$ is a simplicial complex of dimension $2$ such that  $\chi(K) \leq 2$ we will write $\rho(K)$ to mean $\rho(\chi(K))$.
Also, we will denote by $\alpha_i(K)$ the number of $i$-simplices of $K$.
\end{notat}

We prove now a simple result that links the special properties of the simplicial complex $L$ from the statement of Lemma \ref{main_tech} to a lower bound on its number of 2-simplices $\alpha_2(L)$. In what follows we will understand that a simplicial complex of dimension 2 is of strict dimension 2, i.e. it has at least one 2-simplex.

\begin{lemma}\label{euler}(cf. \cite[Lemma 2.2]{BM}).
Let $L$ be a connected simplicial complex of dimension 2 in which every edge is the face of at least two 2-simplices. Then, if $\chi(L) \leq 2$, the complex $L$ has at least $\rho(L)$ vertices and at least $2\alpha_0(L) - 2 \chi(L)$ 2-simplices.
\end{lemma}

\begin{proof}
Consider the Euler characteristic formula,
\[
	\chi(L) = \alpha_0(L) - \alpha_1(L) + \alpha_2(L).
\]
Since every edge of $L$ is the face of at least two 2-simplices, we see that $3\alpha_2(L) \geq 2\alpha_1(L)$. On the other hand, since $L$ is a simplicial complex it has at most $\binom{\alpha_0(L)}{2}$ edges. 
Then
\[
	6\chi(L) \geq 6\alpha_0(L) - \alpha_0(L)(\alpha_0(L) - 1).
\]
If $\chi(L) \leq 0$, the minimum strictly positive integer that satisfies this inequality is precisely $\rho(L) = \rho(\chi(L))$ and therefore $\alpha_0(L) \geq \rho(L)$. An easy analysis shows that $\alpha_0(L) \geq \rho(L)$ also when $\chi(L) = 1,2$. Finally, the claimed lower bound $\alpha_2(L) \geq 2\alpha_0(L) - 2 \chi(L)$ follows immediately from the Euler characteristic formula and the inequality $3\alpha_2(L) \geq 2\alpha_1(L)$.
\end{proof}

We are now ready to prove the main result of this section.

\begin{theo}\label{lower_bound}
Let $G$, $T$ be finitely presented groups. If $G$ satisfies property (A), $\chi_{\leq 2}(G)\leq 2$, and $\dim H_2(G) > 0$, then $\kappa(G \ast T) \geq 2\rho(\chi_{\leq 2}(G)) - 2 \chi_{\leq 2}(G)$.
\end{theo}

\begin{proof}
Let $K$ be a simplicial complex of dimension 2 with fundamental group isomorphic to $G \ast T$. Since $G$ satisfies property (A), from Lemma \ref{main_tech} we obtain a simplicial complex $L$ with $\alpha_2(L) \leq \alpha_2(K)$, $\chi(L) \leq \chi_{\leq 2}(G)$ and such that every edge of $L$ is in at least two 2-simplices. Furthermore, there is an epimorphism $H_2(L) \to H_2(G)$, so that $\dim H_2(L) > 0$ and hence $L$ is of dimension 2. By Lemma \ref{euler}, $L$ has at least $2\rho(L) - 2\chi(L)$ 2-simplices. Now, since $\chi(L) \leq \chi_{\leq 2}(G)$ and $\rho$ is a non-increasing function, we conclude that $\alpha_2(L) \geq 2\rho(\chi_{\leq 2}(G)) - 2\chi_{\leq 2}(G)$, as desired.
\end{proof}

We may apply Theorem \ref{lower_bound} to the one-relator groups from Example \ref{rem_surfaces}. For instance, the theorem gives the lower bound $\kappa(BS(m,n)) \geq 14$ for Baumslag-Solitar groups with $m$, $n$ odd since $\chi(BS(m,n)) = 0$. We know that this bound is not sharp except for the fundamental group of the torus $\Z \oplus \Z = BS(1,1)$. But one would expect stronger lower bounds for one-relator groups with a large number of generators (and hence, small Euler characteristic). As we will see in the next section, the lower bound from Theorem \ref{lower_bound} is sharp for fundamental groups of surfaces.

\section{The simplicial complexity of surface groups}

In this section we derive the main results of the article. Recall that the number of 2-simplices in a minimal triangulation of a closed surface was computed by Ringel \cite{Rin}, in the non-orientable case, and by Jungerman 
and Ringel \cite{JR}, in the orientable case. 

\begin{theo}\label{jungerman}
(Jungerman and Ringel) Let $S$ be a closed surface different from the orientable surface of genus $2$ ($M_2$), the Klein bottle ($N_2$) and the non-orientable surface of genus $3$ ($N_3$). Then, we have the following formula for the number $\delta(S)$ of 2-simplices in a minimal triangulation of $S$:
\[
	\delta(S) = 2\rho(\chi(S)) - 2\chi(S).
\]
For the exceptional cases $M_2, N_2$ and $N_3$, it is necessary to replace $\delta(S)$ by $\delta(S) - 2$ in the formula.
\end{theo}

As it was observed in Example \ref{rem_surfaces}, the fundamental group of a non-simply connected closed surface $S$ satisfies property (A). Hence, we may apply Theorem \ref{lower_bound} to groups of the form $\pi_1(S) \ast T$ obtaining the following corollary.

\begin{prop}\label{non_excep}
Let $S$ be a non-simply connected closed surface. Then $\kappa(\pi_1(S) \ast T) \geq 2\rho(\chi(S)) - 2\chi(S)$. In particular, if $S$ is non-exceptional, then $\kappa(\pi_1(S) \ast T) \geq \delta(S)$.
\end{prop}

\begin{proof}
By Theorem \ref{lower_bound}, we have that $\kappa(\pi_1(S) \ast T) \geq 2\rho(\chi_{\leq 2}(\pi_1(S))) - \chi_{\leq 2}(\pi_1(S))$. Hence, it is enough to see that $\chi_{\leq 2}(\pi_1(S)) = \chi(S)$ for all non-simply connected closed surfaces $S$. This identity is clear for surfaces $S$ different from the real projective plane $\R P^{2}$ since these surfaces are aspherical. For $\R P^{2}$, notice that the infinite real projective space $\R P^{\infty}$ is an Eilenberg-MacLane space for $\pi_1(\R P^{2}) = \Z_{2}$, so we have $\chi_{\leq 2}(\pi_1(S)) = \chi(S)$ also for $S = \R P^{2}$.
\end{proof}

It remains to handle the exceptional cases. Observe that for an exceptional surface $S$ (i.e. $S = N_2, N_3$ or $M_2$), Proposition \ref{non_excep} provides the lower bound $\kappa(\pi_1(S) \ast T) \geq 2\rho(\chi(S)) - 2\chi(S)$, which is slightly weaker than required because $\delta(S) = 2 \rho(\chi(S)) - 2 \chi(S) + 2$ in these cases. So, for the exceptional surfaces, we will need to refine the proof of the lower bound of Theorem \ref{lower_bound}.

\begin{lemma}\label{prop_A}
Let $S$ be a non-simply connected closed surface (either exceptional or non-exceptional) and let $K$ be a connected simplicial complex of dimension $2$ with fundamental group isomorphic to $\pi_1(S) \ast T$. Let $L$ be the simplicial complex obtained from $K$ by applying Lemma \ref{main_tech}. If $\chi(L) = \chi(S)$, then $L$ satisfies property (A).
\end{lemma}

\begin{proof}
From the proof of Lemma \ref{main_tech} applied to $K$, we obtain a continuous map $f: L \vee_{i=1}^{m} S^{1} \to K_{\pi_1(S)} \vee K_{T}$ which induces isomorphism in $H_n$ for $n = 0,1$ and such that the inverse of $f^{*}: H^{1}(\pi_1(S)) \times H^{1}(T) \to H^{1}(L) \times H^{1}(\vee_{i=1}^{m} S^{1})$ restricts to give a monomorphism $h: H^{1}(\vee_{i=1}^{m} S^{1}) \to H^{1}(T)$. On the other hand, $\chi(L) = \chi(S) - (\dim H_1(T) - m)$ and since $\chi(L) = \chi(S)$ by assumption, $m = \dim H_1(T)$ and so $h$ is an isomorphism. Take $a \in H^{1}(L)$ a non-trivial class and let $(\alpha, \delta) \in H^{1}(\pi_1(S)) \times H^{1}(T)$ be the unique class such that $f^{*}(\alpha, \delta) = (a,0)$. Notice that $\alpha \neq 0$. Indeed, $h$ being an isomorphism, the equality $\alpha = 0$ would contradict the injectivity of the inverse of $f^{*}: H^{1}(\pi_1(S)) \times H^{1}(T) \to H^{1}(L) \times H^{1}(\vee_{i=1}^{m} S^{1})$. Hence, there exists a non-trivial class $\beta \in H^{1}(\pi_1(S))$ with $\alpha \cup \beta \neq 0 \in H^{2}(\pi_1(S))$. Put $(b,\lambda) = f^{*}(\beta, 0)$. Note that $f^{*}(\alpha \cup \beta, 0)\neq 0$ (see the proof of Lemma \ref{main_tech}).  By naturality of the cup product,
\[
	0 \neq f^{*}(\alpha \cup \beta, 0) = f^{*}((\alpha, \delta) \cup (\beta, 0)) = f^{*}(\alpha,\delta) \cup f^{*}(\beta, 0) = (a,0) \cup (b, \lambda) = (a \cup b, 0).
\]
Then $a \cup b \neq 0$ and hence $L$ satisfies property (A).
\end{proof}

The next result, which was proved in \cite{BM}, states roughly that if the complex $L$ obtained from Lemma \ref{main_tech} additionally satisfies property (A), it is close to being homeomorphic to a surface.

\begin{prop}\label{homeosurf}
Let $K$ be a simplicial complex of dimension $2$ such that each edge of $K$ is the face of exactly two 2-simplices and let $S$ be a closed surface. Suppose that the homology of $K$ is isomorphic to the homology of $S$ and that the cohomology ring of $K$ satisfies property (A). Then $K$ is homeomorphic to $S$.
\end{prop}

For a proof, see \cite[Proposition 3.1]{BM}.

\begin{prop}\label{exceptional}
Let $S = N_2, N_3$ or $M_2$. Then $\kappa(\pi_1(S) \ast T) \geq \delta(S)$.
\end{prop}

\begin{proof}
Let $K$ be a simplicial complex with fundamental group isomorphic to $\pi_1(S) \ast T$. From Lemma \ref{main_tech}, and keeping the notations of the proof of Theorem \ref{lower_bound}, we obtain a complex $L$ with $\alpha_2(L) \leq \alpha_2(K)$, $\chi(L) \leq \chi_{\leq 2}(\pi_1(S)) = \chi(S)$ and such that every edge of $L$ is in at least two 2-simplices. By Lemma \ref{euler}, this implies that $L$ has at least $\rho(L) \geq \rho(S)$ vertices and at least $2\rho(L) - 2\chi(L)$ 2-simplices. Note that if any of the strict inequalities $\alpha_0(L) > \rho(S)$, $\chi(L) < \chi(S)$ holds, we have
\[
\alpha_2(L) \geq 2\rho(S) - 2\chi(S) + 2 = \delta(S)
\]
and there is nothing to prove. In view of this, in what follows we will suppose  that $\alpha_0(L) = \rho(S)$ and $\chi(L) = \chi(S)$. Observe that $L$ satisfies property (A) by Lemma \ref{prop_A}. Also, since $3\alpha_2(L) \geq 2\alpha_1(L)$, by the Euler characteristic formula for $L$ we have
\[
	3(\alpha_0(L) - \chi(L)) \leq \alpha_1(L) \leq \binom{\alpha_0(L)}{2}.
\]
We solve first the case $S = N_2$. By our assumption, we have that $\chi(L) = \chi(N_2) = 0$ and $\alpha_0(L) = \rho(N_2) = 7$. Hence, from the above inequality we learn that $\alpha_1(L) = 21$ and, since $\chi(L) = 0$, $\alpha_2(L) = 14$. Thus $3\alpha_2(L) = 2 \alpha_1(L) = 42$, from where it follows that every edge of $L$ is the face of exactly two 2-simplices. Since the homology of $L$ is isomorphic to the homology of $S$, by Proposition \ref{homeosurf} $L$ would be homeomorphic to $N_2$ contradicting Theorem \ref{jungerman}. Hence, $\alpha_0(L) > \rho(N_2)$ or $\chi(L) < \chi(S)$ and consequently $\alpha_2(L) \geq \delta(N_2)$.
\par For the surface $S = N_3$, we know that $\chi(L) = \chi(N_3) = -1$ and $\alpha_0(L) = \rho(N_3) = 8$. Hence, 
\[
3(\alpha_0(L) - \chi(L)) = 27 \leq \alpha_1(L) \leq 28 = \binom{\alpha_0(L)}{2}.
\]
Suppose first that $\alpha_1(L) = 27$, so that $\alpha_2(L) = 18$. Hence every edge of $L$ is the face of exactly two 2-simplices and from Proposition \ref{homeosurf}, $L$ is homeomorphic to $N_3$ in contradiction to Theorem \ref{jungerman}. Then, $\alpha_1(L) = 28$. In that case, $\alpha_2(L) = 19$ and since $57 = 3\alpha_2(L) = 2\alpha_1(L) + 1$, every edge of $L$ is in two 2-simplices except for one that is the face of three 2-simplices of $L$. The link of a vertex of this edge is a graph in which every vertex has degree two except for one that has degree three. This is impossible because the sum of the degrees of an undirected graph is even.
Therefore, $\alpha_0(L) > \rho(N_3)$ or $\chi(L) < \chi(N_3)$ and hence $\alpha_2(L) \geq \delta(N_3)$ as claimed.
\par Finally, when $S = M_2$ we have that $\chi(L) = \chi(M_2) = -2$ and $\alpha_0(L) = \rho(M_2) = 9$. In this case, we know that $\alpha_2(L) \geq 22 = \delta(M_2) - 2$ and we want to show that $L$ has at least $\delta(M_2) = 24$ 2-simplices. We will see that the cases $\alpha_2(L) = 22$, $\alpha_2(L) = 23$ are not possible. Suppose first that $\alpha_2(L) = 22$. Then, by the Euler characteristic formula, $\alpha_1(L) = 33$. Therefore, every edge of $L$ is the face of exactly two 2-simplices of $L$ and so, by Proposition \ref{homeosurf} $L$ should be homeomorphic to $M_2$, which contradicts Theorem \ref{jungerman}. If $\alpha_2(L) = 23$, it is $\alpha_1(L) = 34$, whence $69 = 3\alpha_2(L) = 2\alpha_1(L) + 1$. It follows that every edge of $L$ is the face of exactly two 2-simplices except for one which is the face of three 2-simplices. The same argument as before shows that this is impossible. We conclude that $\alpha_2(L) \geq \delta(M_2)$.
\end{proof}

We obtain Theorems \ref{kappa_surf} and \ref{kappa_surf_free_prod} as corollaries of the previous propositions.

\begin{proof}[Proof of Theorem \ref{kappa_surf}]
The upper bound $\kappa(\pi_1(S)) \leq \delta(S)$ is clear, while the lower bound follows from Propositions \ref{non_excep} and \ref{exceptional}.
\end{proof}

Note that, as a consequence of this result, the simplicial complexity of surface groups grows linearly on the genus. This was observed, in the orientable case, in \cite[Example 2]{BBB}.

\begin{proof}[Proof of Theorem \ref{kappa_surf_free_prod}]
Let $T$ be a finitely presentable group. By Propositions \ref{non_excep} and \ref{exceptional}, $\kappa(\pi_1(S) \ast T) \geq \delta(S)$ and since $\kappa(\pi_1(S)) = \delta(S)$ by Theorem \ref{kappa_surf}, the first claim holds. For the second one, it is enough to observe that the upper bound $\kappa(G \ast \Z) \leq \kappa(G)$ holds trivially for every finitely presentable group $G$. 
\end{proof}

Arguably, the conclusions of Theorems \ref{kappa_surf} and \ref{kappa_surf_free_prod} should hold at least for finitely presentable groups $G$ satisfying property (A) in place of surface groups. Unfortunately, we have not been able to establish these results in the general case.

\end{document}